\documentclass[11pt]{article}

\usepackage{amsmath,amssymb,amsthm,amscd}
\usepackage{calrsfs,enumitem}
\usepackage[dvipsnames]{xcolor}
\usepackage{mathtools}
\usepackage{graphicx}
\usepackage[margin=1in]{geometry}

\usepackage[linktocpage]{hyperref}
\hypersetup{
    colorlinks=true,
    linkcolor=red,
    hypertexnames=false,
    filecolor=black,      
    urlcolor=blue,
    citecolor=blue
}

\numberwithin{equation}{section}

\theoremstyle{plain}
\newtheorem{theorem}{Theorem}[section]
\newtheorem*{nthm}{Theorem}
\newtheorem{maintheorem}{Theorem}

\newtheorem{lemma}[theorem]{Lemma}
\newtheorem{corollary}[theorem]{Corollary}
\newtheorem*{conjecture}{Conjecture}

\theoremstyle{definition}

\theoremstyle{remark}
\newtheorem{remark}[theorem]{Remark}

\def\Ga{\Gamma}
\def\ga{\gamma}
\def\ra{\rightarrow}
\def\B{\mathbf{B}}
\def\C{\mathbb{C}}
\def\H{\mathbf{H}}
\def\N{\mathbb{N}}

\def\R{\mathbb{R}}

\DeclareMathOperator{\PU}{PU}

\title{Complex hyperbolic Kleinian groups of large critical exponents}

\author{Subhadip Dey \and Beibei Liu}
\begin{document}
\maketitle
\begin{abstract}
In this article, we show that there exist discrete isometry groups of the $2$- and $3$-dimensional complex hyperbolic spaces with critical exponents arbitrarily close to but strictly smaller than the maximum possible value. This result shows no gap in the values of critical exponents for complex hyperbolic Kleinian groups.  
\end{abstract}

%{\color{red} Note: Still need to figure out the finitely generated issue. It might be correct but requires some additional argument and references.}

\section{Introduction}
The study of discrete subgroups of semisimple Lie groups has been very fruitful in the last 50 years due to a series of rigidity results, starting with the Mostow rigidity theorem for finite volume hyperbolic manifolds of dimensions at least $3$ \cite{Mostow1} and its extension to other finite volume locally symmetric spaces \cite{Mostow2,Prasad}.
A deep consequence of the
Margulis superridigity theorem \cite{Margulis}, Corlette archimedean superrigidity theorem \cite{Corlette1}, and Gromov--Schoen nonarchimedean superrigidity theorem \cite{gromov1992harmonic}  is that the irreducible lattices in semisimple Lie groups (of noncompact type and whose Lie algebras do not contain $\mathfrak{so}(m,1)$ and $\mathfrak{su}(m,1)$ simple factors) are arithmetic. 
To the contrary, due to Gromov and Piatetski-Shapiro's \cite{GPS} hybrid constructions, nonarithmetic lattices exist in the isometry groups of real hyperbolic spaces of all dimensions.
In the complex hyperbolic case, nonarithmetic lattices also exist, but only a handful of examples are known, primarily due to Deligne--Mostow \cite{MR1241644}.
It is a central open question in the field whether complex hyperbolic nonarithmetic lattices exist in all dimensions.
This question gives evidence that the complex hyperbolic geometry is a somewhat elusive demarcation line between the relatively flexible structures of real hyperbolic geometries  and the super-rigid structures of the quaternionic hyperbolic, Cayley, and higher-rank side. 

Another rigidity principle that the semisimple Lie groups of noncompact type (whose Lie algebras do not contain $\mathfrak{so}(m,1)$ and $\mathfrak{su}(m,1)$ factors) share is  Kazhdan's Property (T) \cite{MR0209390,bekka2008kazhdan}. 
%From this point of view, ${\rm PO}(m,1)$, the group of orientation preserving isometries of $m$-dimensional real hyperbolic space, and $\PU(m,1)$, group of holomorhic isometries of the $m$-dimensional complex hyperbolic space, are {\em flexible}, i.e., they do not have Property (T).
From this point of view, ${\rm PO}(m,1)$ and $\PU(m,1)$, the isometry groups of the $m$-dimensional real and complex hyperbolic spaces, respectively, are {\em flexible}, i.e., they do not have Property (T).
For the isometry groups of the quaternionic hyperbolic spaces and Cayley plane, Corlette \cite{Corlette2} noted that Property (T) restricts the possible ``sizes'' of  the limit sets of discrete subgroups $\Ga$.
Appealing to the classical {\em Patterson-Sullivan theory} \cite{Patterson,Sullivan2}, Corlette showed that the {\em critical exponent} $\Ga$, denoted by $\delta_\Ga$, gives a measurement for the size of its conical limit set.
He discovered a surprising gap phenomenon:

\begin{nthm}[Corlette]\leavevmode
\begin{enumerate}[label=(\roman*)]\itemsep0em
\item If $\Ga$ is a discrete isometry group of the $m$-dimensional quaternionic hyperbolic space, then either $\delta_{\Ga}=4m+2$ or $\delta_{\Ga}\leq 4m$. The subgroup $\Ga$ is a lattice if and only if $\delta_{\Ga}=4m+2$.
\item If $\Ga$ is a discrete isometry group of the hyperbolic Cayley plane, then $\delta_{\Ga}=22$ or $\delta_{\Ga}\leq 16$. The subgroup $\Ga$ is a lattice if and only if $\delta_{\Ga}=22$. 
\end{enumerate}
\end{nthm}

There are some striking consequences of this result: For instance, in the quaternionic hyperbolic or Cayley geometry, {\em geometrically infinite} isometry groups are {\em small} in the sense of the critical exponents. 
Such phenomenon does not occur in real hyperbolic spaces; see \cite[Theorem 1.2]{BJ}.
Furthermore, the gap itself is absent for real hyperbolic spaces as Sullivan \cite{Sullivan1} constructed a sequence of convex cocompact isometry groups of the real hyperbolic space of dimension 3 whose critical exponents are arbitrarily close to $2$. See also \cite[Section 4]{Kapovich:2021ud} for such examples in real hyperbolic spaces of all dimensions.

In this article, we focus on the somewhat more mysterious complex hyperbolic spaces.
Since the Lie group $\PU(m,1)$, which is the group of holomorphic isometries of the complex hyperbolic $m$-space $\H_\C^m$, does not have Property (T), the analog of Corlette's gap theorem stated above has no immediate support.
On the other hand,
Sullivan's construction is hard to perform in the complex hyperbolic geometry, as there are no totally geodesic (real) codimension-one submanifolds of the complex hyperbolic spaces. 

Complex hyperbolic geometry, however, shares the rich features of K\"ahler geometry. 
In this article, we show that if there exists a compact complex hyperbolic $m$-manifold, %%orbifolds 
which admits a non-constant holomorphic map  to a compact hyperbolic Riemann surface, then one can construct discrete subgroups of $\PU(m,1)$ with critical exponents arbitrarily close to but strictly smaller than the maximum possible value, $2m$. See \S\ref{sec:construction} and \S\ref{sec:proof} for more details. 
%Such examples were first constructed by Livne \cite{Livne} in his Ph.D. thesis.
Since examples of such manifolds are known in dimensions $m=2,3$ (\cite{Livne,MR2753711}), we obtain the following result:

\begin{maintheorem}\label{mainthm}
Let $m=2$ or $3$.
Then, there exists a sequence $(\Gamma_n)$ of finitely generated
discrete isometry groups  of $\H_{\C}^{m}$ such that: 
\begin{enumerate}[label=(\roman*)]\itemsep0em
 \item For all $n\in\N$, the critical exponent $\delta_{\Ga_n}$ is strictly lesser than $2m$, and
 \item $\lim_{n\rightarrow \infty}\delta_{\Ga_n}=2m$.
\end{enumerate}
\end{maintheorem}

A consequence of this result is that Corlette's gap theorem stated above has no analog in the complex hyperbolic geometry, at least in dimensions $2$ and $3$.
In this regard, we highlight a related conjecture of Yue \cite{Yue}, which was inspired by Corlette's gap theorem: 

\begin{conjecture}[Yue]
Suppose that $\Ga$ is a convex cocompact isometry group of $\H_{\C}^{m}$. Then $\Ga$ is either a uniform lattice (thus, $\delta_{\Ga}=2m$), or $\delta_{\Ga}\leq 2m-1$. 
\end{conjecture}

However, Theorem \ref{mainthm} does not disprove this conjecture, as the discrete groups $\Ga_n$  we construct here are all geometrically infinite.
See Remark \ref{final_remark}.

\subsection*{Acknowledgement} We  would like to thank Igor Belegradek,  Christopher J. Bishop,  John Etnyre, and Shi Wang for helpful discussions. We are especially grateful to Misha Kapovich for suggesting us to study Livne's example and  sharing his ideas with us. 
We thank our referees for helping us improve the exposition and suggesting 
Deraux's article which allows us to include the case $m=3$ in our main result.
The second author is partially supported by the NSF grant DMS-2203237. 

\section{Preliminaries}
Let $\B^m$ denote the unit ball in the $m$-dimensional complex vector space $\C^m$.
The group of biholomorphic automorphisms of $\B^m$ is naturally identified with the projective unitary group $\PU(m,1)$.
It carries a natural $\PU(m,1)$-invariant Hermitian metric of constant holomorphic sectional curvature, called the {\em Bergman metric}.
The ball $\B^m$ equipped with the real part of the Bergman metric is called the $n$-dimensional complex hyperbolic space, denoted by $\H_\C^m$.
Although $\H_\C^m$ is a negatively-curved space, unlike the real hyperbolic case, the  real sectional curvature is variable.
Under a suitable normalization, which we use throughout this article,
the real sectional curvature varies between $-4$ and $-1$.
We refer to Goldman's book \cite{MR1695450} and Kapovich's survey \cite{kapovich2019lectures} for a detailed discussion of complex hyperbolic geometry.
Furthermore, Kapovich's survey \cite{kapovich2019lectures} contains a collection of important open questions in this field.

Discrete subgroups $\Ga$ of $\PU(m,1)$ act properly discontinuously on $\H_\C^m$.
To each discrete subgroup $\Ga$, one associates a numerical invariant, called the {\em critical exponent}, as follows: Fixing a base point $x_0\in\H_\C^m$, consider the Poincar\'e series
\[
 P(s) = \sum_{\ga\in\Ga} \exp( -s\, d_{\H_\C^m}(x_0,\ga x_0)).
\]
The {\em critical exponent} of $\Ga$, denoted by $\delta_\Ga$, is
\[
\delta_\Ga = \inf\{s\in [0,\infty) : P(s) \text{ converges} \}.
\]
A fundamental fact is that the number $\delta_\Ga$ defined above is independent of the choice of the base point $x_0 \in \H_\C^m$. For details, see for example \cite[Section 7]{kapovich2019lectures}.

The critical exponent is deeply connected with other invariants associated with $\Ga$: for instance,  Corlette \cite{Corlette2} showed that the critical exponent of a  {\em geometrically finite} group $\Ga$ equals  the Hausdorff dimension of limit set of $\Ga$ equipped with  the restriction of certain natural sub-Riemannian  metric (called a {\em Carnot metric}) on the visual boundary of $\H^m_\C$.
More generally, Corlette--Iozzi \cite{MR1458321} showed that for general discrete groups $\Ga$, $\delta_\Ga$ coincides with the Hausdorff dimension (w.r.t. the Carnot metric) of the {\em conical} part of the limit set of $\Ga$.

Another invariant that the critical exponent has deep connection with, which plays a central role in this article, is the spectrum of the {\em Laplace-Beltrami operator}. For a discrete subgroup $\Ga < \PU(m,1)$, let $M_\Ga \coloneqq \H_\C^m/\Ga$, and let $\lambda_0(M_\Ga)$ denote the infimum of the spectrum of the Laplace-Beltrami operator $\Delta_{M_\Ga}$ on $L^2(M_\Ga)$.

\begin{theorem}[Elstrodt-Patterson-Sullivan-Corlette; see {\cite[Theorem 4.2]{Corlette2}}]
Let $\Ga < G = \PU(m,1)$ be a discrete group.
 Then
 \[
 \lambda_0(M_\Ga) = 
 \begin{cases} 
      m^2 & \text{if }\delta_\Ga\in [0,m], \\
      \delta_\Ga(2m - \delta_\Ga) & \text{if }\delta_\Ga\in [m,2m].
 \end{cases}
 \]
\end{theorem}

A direct consequence, which is useful for our purpose, is the following:

\begin{corollary}\label{cor}
 Let $(\Ga_n)$ be a sequence of discrete subgroups of $\PU(m,1)$, and let $M_n \coloneqq \H_\C^m/\Ga$. If $\lambda_0(M_{n})> 0$ and $\lim_{n\ra\infty} \lambda_0(M_{n}) = 0$, then $\delta_{\Ga_n} <2m$ and
$
\lim_{n\ra\infty} \delta_{\Ga_n} = 2m.
$
\end{corollary} 

\section{Main construction}\label{sec:construction}

In this section, we present a general construction of $m$-dimensional complex hyperbolic Kleinian groups whose critical exponents approach to  arbitrarily close to $2m$, the maximum possible value. However, to start the construction, one needs a compact  complex hyperbolic $m$-manifold $M$ admitting a non-constant holomorphic map 
$
  \psi: M \ra S
$
to a compact hyperbolic Riemann surface $S$.

Let $m\ge 2$ and suppose that we are given  such an example $\psi: M \ra S$.
We observe that the image of the induced homomorphism $\psi_*: \pi_1(M)\to \pi_1(S)$ is a finite index subgroup of $\pi_1(S)$; hence the image $\psi_*(\pi_1(M))$ is the fundamental group of a closed real hyperbolic surface (i.e., a {\em surface group}):
Indeed, if, to the contrary, $\psi_*(\pi_1(M))$  had an infinite index in $\pi_1(S)$, then we could lift $\psi: M \ra S$ to a non-constant holomorphic map $M \to S'$, where $S' \to S$ is the infinite degree covering space corresponding to the infinite index subgroup $\psi_*( \pi_1(M))< \pi_1(S)$. However, since $M$ is compact and $S'$ is noncompact, such non-constant holomorphic maps cannot exist.

Since $\pi_1(M)$ admits an epimorphism to a surface group, by Siu--Beauville theorem (see \cite[Theorem 2.11]{amoros1996fundamental}), $M$ admits a surjective holomorphic map with connected fibers to a hyperbolic Riemann surface, which we again denote by %which is some finite cover of $S''$,
\begin{equation}\label{eqn:main_example}
  \psi: M \ra S.
\end{equation}
In particular, we obtain a short exact sequence
\begin{equation}\label{eqn:SES}
 1 \to K \to \pi_1(M) \xrightarrow{\psi_*} \pi_1(S) \to 1,
\end{equation}
where $K$ is the kernel of $\psi_*$.
By definition, $\pi_1(M)$ is realized as a cocompact lattice in $\PU(m,1)$.
Then, $K$, as a subgroup of $\pi_1(M)$, also embeds in $\PU(m,1)$ as a discrete subgroup.
Since $K$ is an {\em infinite} normal subgroup of $\pi_1(M)$, 
\begin{equation}\label{eqn:dK}
 \Lambda(K) = \partial\H^m_\C.
\end{equation}
(Here and in what follows, for $\Ga < \PU(m,1)$, $\Lambda(\Ga)$ denotes the {\em limit set} of $\Ga$.)
Let
$\tilde M$
denote the regular covering of $M$ corresponding to the normal subgroup $K< \pi_1(M)$.
Then \eqref{eqn:main_example} induces a $\pi_1(S)$-equivariant surjective  holomorphic map
\begin{equation}\label{eqn:K}
 \tilde\psi: \tilde M \to \H^2_\R
\end{equation}
with connected fibers.
We observe that there is a uniform upper-bound for the diameter of the fibers of $\tilde\psi$.

We now pick a non-separating simple closed geodesic $c$ in $S$.
Cutting $S$ along $c$, we obtain a connected hyperbolic $\hat S$ surface with geodesic boundary $\partial \hat S$. Note that $\partial \hat S$ has two connected components $c_-$ and $c_+$ and there is a natural identification $f: c_- \to c_+$.
Let
\[
 \hat S_n \coloneqq \underbrace{\hat S \sqcup_f \hat S \sqcup_f \cdots \sqcup_f \hat S}_{n-\text{times}},
\]
where the symbol $\sqcup_f$ above is used to mean that two consecutive copies of $\hat S$ appearing on the left
and right of the symbol $\sqcup_f$ are glued along their boundaries $c_-$ (for the left one) and $c_+$ (for the right one) by the homeomorphism $f$.
Also,  $\hat S_n$ is a compact hyperbolic surface with geodesic boundary.
The natural map $\hat S_n \to S$ induces an injective homomorphism $\pi_1(\hat S_n) \hookrightarrow \pi_1(S)$; we will identify $\pi_1(\hat S_n)$ as a subgroup of $\pi_1(S)$ under this homomorphism.
The universal cover of $\hat S_n$ is isometric to a convex domain $D_n\subset\H^2_\R$ with geodesic boundary; the fundamental group $\pi_1(\hat S_n)$, which is the deck-transformation group of the covering $D_n\to \hat S_n$, acts on $D_n$ as a subgroup of $\pi_1(S)$.
Taking quotients by $\pi_1(\hat S_n)$, \eqref{eqn:K} induces a surjective holomorphic map
\[
 \psi_n: M_n \to S_n,
\]
where 
\begin{equation}\label{eqn:Mn}
 M_n = \tilde M/\pi_1(\hat S_n) \quad\text{and}\quad 
 S_n = \H^2_\R/\pi_1(\hat S_n).
\end{equation}

\begin{lemma}\label{lem:qi}
 The map $\psi_n: M_n \to S_n$ is a quasi-isometry.
\end{lemma}

\begin{proof}
 Since $\psi:M\to S$ is $L$-Lipschitz, for some $L\ge 1$, it follows that $\psi_n: M_n \to S_n$ is also $L$-Lipschitz.
 So, it is enough to construct a quasi-inverse of $\psi_n$:
 Using the Axiom of Choice, we get a map $\phi_n: S_n \to M_n$ which sends any point $x\in S_n$ to some point in $\psi_n^{-1}(x)$. Clearly, $\psi_n\circ\phi_n = {\rm id}_{S_n}$.
 Moreover, since the diameter of $\psi_n^{-1}(x)$, $x\in S_n$, is uniformly bounded (cf. last line of the fourth  paragraph in this section), say by some number $D>0$, it follows that 
 \[
  d_{M_n}(y, \phi_n\circ\psi_n(y)) \le D,
 \]
 for all $y\in M_n$. Therefore, $\phi_n$ is a quasi-inverse of $\psi_n$.
\end{proof}

\begin{lemma}\label{lem:one}
 Let $h_n$ denote the Cheeger's isoperimetric  constant of $M_n$.
 Then, $h_n \ra 0$, as $n\ra\infty$.
 Consequently, $\lambda_0(M_n) \ra 0$, as $n\ra\infty$.
\end{lemma}

\begin{proof}
Let $\hat M_n$ denote the preimage of $\hat S_n \subset S_n$ under $\psi_n$.
Note that the volume of $\hat M_n$ equals  $n \,{\rm vol}\,M$.
Moreover, the boundary of the $M_n$, which is the preimage of $\partial \hat S_n$ under $\psi_n$, has two connected components, each one is isometrically identified with $\psi^{-1}(c)$.
Therefore,
 \[
  h_n \le \frac{{\rm area}\, \partial \hat M_n}{{\rm vol}\, \hat M_n}
  = \frac{2\,{\rm area}\, \psi^{-1}(c)}{n \,{\rm vol}\,M} \ra 0,
 \]
 as $n\ra\infty$.
 Then, by Buser's inequality \cite[Theorem 7.1]{MR683635}, $\lambda_0(M_n) \ra 0$, as $n\ra\infty$.
\end{proof}

\begin{lemma}\label{lem:two}
 For all $n\in\N$, $\lambda_0(M_n)>0$.
\end{lemma}

\begin{proof}
 Since $M_n$ and $S_n$ are (nonregular) covers of $M$ and $S$, respectively,  the injectivity radii of $M_n$ and $S_n$ are uniformly bounded from below by that of $M$ and $S$, respectively.
 Since $S_n$ and $M_n$ are quasiisometric (see Lemma \ref{lem:qi}), by \cite[Corollary 18.17]{MR3753580},
 \[
  \lambda_0(M_n) > 0 \iff \lambda_0(S_n) >0.
 \]
 Finally, $S_n$ is a convex cocompact (real) hyperbolic surface of infinite area.
 Thus, the critical exponent of $\pi_1(S_n)<{\rm PSL}(2,\R)$ is strictly smaller than $1$.
 Applying the Elstrodt-Patterson-Sullivan theorem, we obtain that $\lambda_0(S_n)>0$;
 see \cite[Theorem 2.17]{Sullivan}.
 Therefore, by above, $\lambda_0(M_n) > 0$.
\end{proof}

\begin{remark}\label{rem:fg}
 We show that, up to passing to a finite degree cover of $M$, one may choose $S$ and $\psi$ in \eqref{eqn:main_example} so that the group $K = \ker \psi_*$ in \eqref{eqn:SES} is finitely generated:
 
  Let $f: X \to C$ be a surjective holomorphic map with connected fibers, where $X$ is a compact K\"ahler manifold and $C$ is a compact Riemann surface.
  Following Catanese \cite[\S4]{Catanese}, we obtain that $f$ induces an exact sequence
 \[
  \pi_1(F) \to \pi_1(X) \xrightarrow{f_*} \pi_1^{\rm orb}(f) \to 1,
 \]
 where $F$ is a smooth fiber of $f$ and $\pi_1^{\rm orb}(f)$ is the orbifold fundamental group of an associated orbifold structure on $C$.
 Since $F$ is compact, $\pi_1(F)$ and hence its image in $\pi_1(M)$, which we denote by $K$,  is finitely generated.
 Assuming $C$ has genus at least $2$, we can realize $\pi_1^{\rm orb}(f)$ as a uniform lattice in ${\rm PSL}(2,\R)$.
 After passing to a finite index torsion-free subgroup $\Sigma$ of $\pi_1^{\rm orb}(f)$, we thus have a short exact sequence
 \[
  1 \to K \to \Pi \xrightarrow{\psi_*} \Sigma \to 1,
 \]
 where $\Pi = f_*^{-1}(\Sigma)$ is a finite index subgroup of $\pi_1(X)$.
 Let $X'$ be the finite degree cover of $X$ corresponding to the subgroup $\Pi$ of $\pi_1(X)$.
 Since $K$ is finitely generated, 
 it follows from \cite[Theorem 4.3]{Catanese} that the epimorphism $\psi_*$ above is induced by a  surjective holomorphic map (with connected fibers) from $X'$
 to a compact hyperbolic Riemann surface.
\end{remark}

\section{Proof of the main result}\label{sec:proof}
We prove our main result:

\begin{proof}[Proof of Theorem \ref{mainthm}]
 The existence of compact complex hyperbolic $2$-manifolds fibering over a hyperbolic Riemann surface was first shown by Livne \cite{Livne}.
Moreover, Deraux \cite[Theorem 3.1(iv,v)]{MR2753711} showed that many  Deligne--Mostow $2$-ball quotients admit non-constant holomorphic maps onto some hyperbolic Riemann surfaces, and so does one of the Deligne-Mostow $3$-ball quotients.
Therefore, for $m=2,3$, there exists a compact complex hyperbolic $m$-manifold  $M$ admitting a surjective holomorphic map with connected fibers to a hyperbolic Riemann surface; see \S\ref{sec:construction}.
Let $M_n$, $n\in\N$, be the  complex hyperbolic $m$-manifold constructed in the previous section (see \eqref{eqn:Mn}).
Define $\Ga_n \coloneqq \pi_1(M_n) < {\rm PU}(m,1)$.
Combining Lemmata \ref{lem:one}, \ref{lem:two}, and Corollary \ref{cor},
 we obtain that,
 \[
 \delta_{\Ga_n} < 2m, \quad\text{for all }n\in\N,
 \]
 and
 \[
  \lim_{n\ra \infty} \delta_{\Ga_n} = 2m.
 \]
 Finally, it follows from the discussion in \S\ref{sec:construction} that for all $n\in\N$, the group $\Ga_n$ fits into a short exact sequence, 
 \[
 1\to K \to \Ga_n \to \pi_1(\hat S_n) \to 1,
 \]
 cf. \eqref{eqn:Mn}.
 Note that $\pi_1(\hat S_n)$ is finitely generated.
 Moreover, by Remark \ref{rem:fg}, we may (and will) also assume that $K$ is finitely generated.
 Therefore, $\Ga_n$ is finitely generated.
\end{proof}

\begin{remark}\label{final_remark}
 As a final remark, we verify that the subgroups $\Ga_n<{\rm PU}(m,1)$ in the proof above are not geometrically finite:
 Indeed, each $\Ga_n$ contains a copy of the group $K$. Since $\Lambda(K) = \partial\H^m_\C$ (see \eqref{eqn:dK}),
 $\Lambda(\Ga_n) = \partial\H^m_\C$.
 However, we also obtained above that $\delta_{\Ga_n} < 2m$.
 Geometrically finite subgroups of $\PU(m,1)$ cannot satisfy both properties simultaneously; see \cite{Corlette2}.
\end{remark}

\bigskip
Department of Mathematics,
  Yale University, 
   219 Prospect St, New Haven, CT 06511
   
subhadip.dey@yale.edu

\medskip
Department of Mathematics,
    The Ohio State University, 100 Math Tower, 231 West 18th Avenue,
   Columbus, OH 43210
   
   bbliumath@gmail.com

\end{document}